\documentclass[11pt]{article}
\usepackage{amsfonts, amsmath, amssymb, amscd, amsthm, color, graphicx, mathrsfs, mathabx, wasysym, setspace, mdwlist, calc,float}
\usepackage{setspace}
\usepackage{hyperref}
\usepackage{tikz-cd}
\usepackage{hyperref}
\usepackage{tocloft}
\usepackage{xcolor}

\usepackage{hyperref}
\hypersetup{linktocpage}

\hypersetup{colorlinks,
    linkcolor={red!50!black},
    citecolor={blue!80!black},
    urlcolor={blue!80!black}}
\usepackage{float}

\renewcommand{\bar}{\widebar}

\newcommand{\e}{\varepsilon}
\newcommand{\NN}{\mathbb{N}}

\newcommand{\ZZ}{\mathbb{Z}}

\renewcommand{\d}{{\rm d}}

\newcommand{\Aut}{\operatorname{Aut}}

\newtheorem{thm}{Theorem}[section]
\newtheorem{cor}[thm]{Corollary}
\newtheorem{lem}[thm]{Lemma}
\newtheorem{prop}[thm]{Proposition}
\newtheorem{prob}[thm]{Problem}

\theoremstyle{definition}
\newtheorem{defn}[thm]{Definition}

\theoremstyle{remark}
\newtheorem{rem}[thm]{Remark}
\newtheorem{ex}[thm]{Example}

\renewcommand{\L}{\mathcal{L}}

\newcommand{\la}{\langle}
\newcommand{\ra}{\rangle}
\newcommand{\Lo}{\L_{\omega_1, \omega}}
\newcommand{\Mod}{{\rm Mod}}

\newcommand\blfootnote[1]{%
  \begingroup
  \renewcommand\thefootnote{}\footnote{#1}%
  \addtocounter{footnote}{-1}%
  \endgroup
}

\begin{document}

\title{Generic length functions on countable groups}
\date{}
\author{A. Jarnevic, D. Osin\thanks{The second author has been supported by the NSF grant DMS-1853989.}, K. Oyakawa}

\maketitle

\vspace{-3mm}

\begin{abstract}
Let $L(G)$ denote the space of integer-valued length functions on a countable group $G$ endowed with the topology of pointwise convergence. Assuming that $G$ does not satisfy any non-trivial mixed identity, we prove that a generic (in the Baire category sense) length function on $G$ is a word length and the associated Cayley graph is isomorphic to a certain universal graph $U$ independent of $G$. On the other hand, we show that every comeager subset of $L(G)$ contains $2^{\aleph_0}$ ``asymptotically incomparable" length functions.  A combination of these results yields $2^{\aleph_0}$ pairwise non-equivalent regular representations $G\to \Aut(U)$. We also prove that generic length functions are virtually indistinguishable from the model-theoretic point of view. Topological transitivity of the action of $G$ on $L(G)$ by conjugation plays a crucial role in the proof of the latter result.
\end{abstract}

\section{Introduction}

Metrics on groups are one of the main objects of study in geometric group theory. These metrics are usually assumed to be invariant under the group action on itself by left multiplication and, therefore, can be encoded in terms of length functions.

\begin{defn}\label{Def:LF}
An (integer-valued) \emph{length function} on a group $G$ is a map $\ell\colon G\to \NN\cup \{ 0\}$ satisfying the following conditions for all $g,h\in G$:
\begin{enumerate}
\item[(L$_1$)] $\ell(g)=0$ if and only if $g=1$;
\item[(L$_2$)] $\ell(g)=\ell(g^{-1})$;
\item[(L$_3$)] $\ell (gh)\le \ell(g)+\ell(h)$.
\end{enumerate}
\end{defn}

Every length function $\ell$ on $G$ gives rise to a left-invariant metric defined by the formula $\d_\ell(a,b)=\ell(a^{-1}b)$ for all $a,b\in G$. It is easy to see that the map $\ell\mapsto \d_\ell$ is a one-to-one correspondence between lengths functions and integer valued left-invariant metrics on $G$.

Group actions on metric spaces and group embeddings serve as a rich source of length functions. The following particular class of examples will play an important role in our paper.

\begin{defn}
A length function $\ell$ on a group $G$ is a \emph{word length} if there exists a generating set $X$ of $G$ such that, for every $g\in G$, $\ell(g)$ equals the length of a shortest word in the alphabet $X\cup X^{-1}$ representing $g$.
\end{defn}

We denote by $L(G)$ the space of length functions on a group $G$ endowed with the topology of pointwise convergence. Recall that a topological space is said to be \emph{perfect} if it has no isolated points. The starting point of our work is the following elementary observation.

\begin{prop}\label{Prop:P}
For every countably infinite group $G$, $L(G)$ is a perfect Polish space.
\end{prop}

In particular, $L(G)$ satisfies the Baire category theorem. Following the standard terminology, we say that a certain property \emph{holds for a generic length function on $G$} if it holds for every length function from a comeager subset of $L(G)$. The goal of this paper is to study geometric and model-theoretic properties of generic length functions on countably infinite groups. To state our main results, we need some auxiliary definitions.

\begin{defn} A group $G$ satisfies a \emph{non-trivial mixed identity} if there exists a non-trivial element $w$ of the free product $G\ast \ZZ$ such that $\alpha(w)=1$ for every homomorphism $\alpha\colon G\ast \ZZ\to G$ satisfying $\alpha(g)=g$ for all $g\in G$. If $G$ does not satisfy any non-trivial mixed identity, it is said to be \emph{mixed identity free} (or \emph{MIF} for brevity).
\end{defn}

The property of being MIF is much stronger than being identity free. On the other hand, the class of MIF groups is rather wide. For instance, it includes all non-cyclic torsion-free hyperbolic groups (e.g., free groups of rank at least $2$ and fundamental groups of closed surfaces of genus at least $2$) and, more generally, all acylindrically hyperbolic groups without non-trivial finite normal subgroups. We refer the reader to Section \ref{Sec:Prelim} for more details and examples.

Our first result shows that generic length functions on countable MIF groups are indistinguishable from the purely geometric point of view. More precisely, for every $\ell\in  L(G)$, we consider the \emph{(unoriented, unlabeled) Cayley graph} with the vertex set $G$, where two vertices $a,b\in G$ are connected by an edge if and only if $\ell(a^{-1}b)=1$. If $\ell$ is a word length, this definition yields the usual Cayley graph with respect to the generating set $X_\ell=\{ g\in G \mid \ell(g)=1\}$. Further, we denote by $U$ the universal graph for isometric embeddings, whose existence was proved by Moss in \cite{M}. Recall that $U$ can be defined as the unique countable graph satisfying the following properties: $U$ is connected, isometrically embeds every finite connected graph, and any isometry between any two finite sets of vertices (endowed with the induced metric) extends to an automorphism of $U$.

\begin{thm}\label{main1}
For any countable MIF group $G$, a generic length function on $G$ is a word length and the corresponding Cayley graph is isomorphic to the Moss graph $U$.
\end{thm}

It is worth noting that the class of countable groups satisfying the conclusion of Theorem \ref{main1} is strictly larger than the class of countable MIF groups but does not contain all countably infinite groups; for detais, see Remark \ref{Rem:ab} and Example \ref{Ex:ab}.

It is also interesting to compare Theorem \ref{main1} with the result of Cameron and Johnson \cite{CJ} stating that, for any countable group $G$ from a large class including all MIF groups, the Cayley graph of $G$ with respect to a randomly chosen generating set is isomorphic to the Rado graph. Our ``randomness model" is different and the outcome is very different as well: the graph $U$ is unbounded as a metric space while the Rado graph has diameter $2$. In particular, we derive the following.

\begin{cor}\label{Cor:main1}
Every countable MIF group has an unbounded Cayley graph.
\end{cor}

Surprisingly, there exist infinite groups such that their Cayley graphs with respect to all generating sets are bounded. For example, $Sym(\NN)$ has this property; moreover, every length function on $Sym(\NN)$ is bounded \cite{B,C}. Whether every infinite \emph{countable} group has an unbounded Cayley graph is an open problem.

In geometric group theory, it is customary to study metrics on groups up to the bi-Lipschitz equivalence relation (for motivation, see \cite[Section 0.3.C]{Gro}). In terms of length functions, the definition can be stated as follows. For two length functions $\ell_1, \ell_2\in L(G)$, we write $\ell_1 \preccurlyeq \ell_2$ if the ratio $\ell_1/\ell_2$ is bounded on $G\setminus \{ 1\}$. Further, $\ell_1$ and $\ell_2$ are \emph{(bi-Lipschitz) equivalent} if $\ell_1 \preccurlyeq \ell_2$ and $\ell_2 \preccurlyeq \ell_1$. If $G$ is finitely generated, then word lengths on $G$ with respect to finite generating sets form a single equivalence class. However, every countably infinite group admits many non-equivalent length functions. More precisely, we prove the following.

\begin{thm}\label{main2}
Let $G$ be a countable group. Every comeager subset of $L(G)$ contains $2^{\aleph_0}$ pairwise $\preccurlyeq$-incomparable (in particular, non-equivalent) length functions.
\end{thm}

Theorem \ref{main1} gives rise to many injective homomorphisms $G\to \Aut(U)$ such that the induced action of $G$ on the set of vertices of $U$ is regular. We call such homomorphisms \emph{regular representations}. One may wonder how different regular representations arising from Theorem \ref{main1} are. As usual, we say that two representations $\alpha_1, \alpha_2\colon G\to \Aut(U)$ are \emph{equivalent} if there exists $t\in \Aut(U)$ such that $$\alpha_1(g)=t^{-1}\alpha_2(g)t$$  for all $g\in G$.

It is not difficult to show that non-equivalent generic length functions yield non-equivalent regular representations $G\to \Aut(U)$ via Theorem \ref{main1}. Combining this with Theorem \ref{main2}, we obtain the following.

\begin{cor}\label{Cor: non-conjugate}
Every countable MIF group $G$ admits $2^{\aleph_0}$ non-equivalent regular representations $G\to \Aut(U)$.
\end{cor}

In particular, every countable group $G$ embeds in $\Aut (U)$ as a subgroup. Indeed, it is straightforward to verify that for every non-trivial countable group $G$, the free product $G\ast \ZZ$ is MIF. Therefore, $G\ast \ZZ$ embeds in $\Aut(U)$. However, there exist countably infinite groups, e.g., the group given by the presentation $\langle a,b\mid b^4=1,\; b^{-1} ab=a^{-1}\rangle$, that do not admit any regular representation in $\Aut(U)$ (see Example \ref{Ex:vc}).

Finally, we study the first order theory of generic length functions on a given group $G$. For this purpose, we consider the first order language $\L$ in the signature
\begin{equation}\label{Eq:sig}
\{\cdot, L_0, L_1, \ldots\},
\end{equation}
where $\cdot$ is a binary operation and each $L_i$ is a unary predicate. Every $\ell\in  L(G)$ gives rise to an $\L$-structure $G_\ell$ as follows. The universe of $G_\ell$ is the group $G$, the operation $\cdot$ is interpreted as the group multiplication, and for any $i\in \NN\cup\{0\}$ and $g\in G$, we have $G_\ell \models L_i(g)$ if and only if $\ell(g)=i$. By abuse of terminology, we say that length functions $\ell_1,\ell_2\in  L(G)$ are \emph{elementarily equivalent}, if so are the $\L$-structures $G_{\ell_1}$ and $G_{\ell_2}$.

\begin{thm}\label{main3}
Suppose that $G$ is a countable MIF group. Then $L(G)$ contains a comeager elementary equivalence class.
\end{thm}

Note that every isomorphism between $\L$-structures $G_{\ell_1}$ and $G_{\ell_2}$ is an automorphism of $G$. If $G$ is finitely generated, $\Aut(G)$ is countable and hence so are the isomorphism classes $[\ell_0]=\{ \ell\in  L(G)\mid G_{\ell}\cong G_{\ell_0}\}$  for all $\ell_0\in  L(G)$. In contrast, every countable MIF group admits a length function whose elementary equivalence class is uncountable. Indeed this follows from Theorem~\ref{main3}, Proposition \ref{Prop:P}, and the well-known fact that every comeager subset of a perfect Polish space has the cardinality of the continuum.

In fact, we derive Theorem \ref{main3} from a much stronger zero-one law for $\Lo$-sentences (see Theorem \ref{Thm:0-1}). The crucial ingredient of the proof of the latter result is topological transitivity of a natural action of a MIF group $G$ on $L(G)$, which seems to be of independent interest. More precisely, for every $\ell\in  L(G)$ and every $g\in G$, let $g\circ \ell$ be the length function on $G$ given by the formula
\begin{equation}\label{Eq:CA}
(g\circ \ell) (x)=\ell(g^{-1}xg)\;\;\; \forall\, x\in G.
\end{equation}
This defines an action $G\curvearrowright L(G)$ by homeomorphisms, which we call the \emph{conjugacy action}.

\begin{thm}\label{main4}
For every countable MIF group $G$, the conjugacy action of $G$ on $L(G)$ is topologically transitive.
\end{thm}

In general, Theorems \ref{main1}, \ref{main3}, and \ref{main4} fail for non-MIF groups. It is natural to ask whether these results hold for groups from certain particular classes of interest. For example, it would be interesting to classify solvable groups satisfying any of these theorems. For some partial results, see Proposition \ref{Prop:non tt} and Corollary \ref{Cor:Sol}.

The paper is organized as follows. In the next section we review the necessary background and prove some preliminary results, including Proposition \ref{Prop:P}. Theorems \ref{main1} and \ref{main2} are proved in Section \ref{Sec:CG}. Section \ref{Sec:TT} is devoted to the question of whether the action of a given group $G$ on $L(G)$ is topologically transitive; in particular, we prove Theorem \ref{main4} there. The proof of Theorem \ref{main3} is given in Section~\ref{Sec:EE}.

\section{Preliminaries}\label{Sec:Prelim}

\paragraph{2.1. Mixed identities.}
Let $G$ be an arbitrary group. We adopt the notation $w(x)$ for elements of the free product $G\ast\ZZ=G\ast \langle x\rangle$ and think of these elements as words in the alphabet $G\cup\{ x^{\pm 1}\}$. For $w(x)\in G\ast \langle x\rangle$ and an element $g\in G$, we denote by $w(g)$ the element of $G$ obtained by substituting $x=g$ in the word $w(x)$. Equivalently, $w(g)$ is the image of $w(x)$ under the evaluation homomorphism $G\ast \langle x\rangle \to G$ that is identical on $G$ and sends $x$ to $g$.

Let $w(x)\in G\ast \langle x\rangle$. An element $g\in G$ \emph{satisfies the equation} $w(x)=1$ if $w(g)=1$. Further, $G$ satisfies the \emph{mixed identity} $w(x)=1$ if $w(g)=1$ for all $g\in G$; this mixed identity is said to be \emph{non-trivial} if $w(x)\ne 1$ as an element of $G\ast \langle x\rangle$. If $G$ does not satisfy any non-trivial mixed identity, it is said to be \emph{mixed identity free} (or \emph{MIF} for brevity).

\begin{ex}
Suppose that $G$ contains a non-trivial element $a$ whose conjugacy class consists of $n$ elements. Then $G$ satisfies the non-trivial mixed identity $[x^{n!},a]=1$.
\end{ex}

All acylindrically hyperbolic groups without non-trivial finite normal subgroups are MIF by \cite[Corollary 5.10]{HO}. We do not discuss acylindrical hyperbolicity in this paper and refer the interested reader to \cite{Osi18, Osi16} and references therein. Instead of giving the definition, we just mention some particular examples of MIF groups from this class (for details, see \cite{HO} and the papers cited below):
\begin{enumerate}
\item[(a)] all torsion-free non-cyclic hyperbolic groups (e.g., free groups $F_n$ of rank $n\ge 2$ and fundamental groups of closed surfaces of genus at least $2$);
\item[(b)] $Out(F_n)$ for $n\ge 3$ \cite{BF};
\item[(c)] mapping class groups of closed surfaces of genus at least $3$ \cite{Bow};
\item[(d)] groups given by presentations with one relation and at least $3$ generators \cite{Osi15}.
\end{enumerate}
Furthermore, combining \cite[Proposition 5.15]{HO} with \cite[Poposition 1.3]{MOW}, one can obtain finitely generated MIF groups of completely different nature, e.g., torsion or simple. Finally, we mention that there exist MIF amenable groups \cite{Jac}; more precisely, these groups are (locally nilpotent)-by-cyclic.

It is often convenient to use a characterization of MIF groups in terms of their universal theory. Recall that a \emph{$G$-group} is a group $H$ with a given embedding $G\le H$. Let $Th_\forall^G(H)$ denote the set of all universal sentences in the standard group-theoretic language with constants from the subgroup $G$ that hold true in $H$.  Two $G$-groups $H_1$, $H_2$ are \emph{universally equivalent} if $Th_\forall^G(H_1)=Th_\forall^G(H_2)$. The following elementary result can be found in  \cite[Proposition 5.3]{HO}.

\begin{prop}\label{Prop:HO}
A countable group $G$ is MIF if and only if $G$ and $G\ast \langle x\rangle$ are universally equivalent as $G$-groups.
\end{prop}

\paragraph{2.2. Polish spaces and the Baire category theorem.}
A topological space is \emph{Polish} if it is separable and completely metrizable. A basic example is the power set $2^\NN$ endowed with the product topology (or, equivalently, the topology of pointwise convergence of indicator functions). More generally, the class of Polish spaces is closed under taking countable products.

A\emph{ $G_\delta$-subset} (respectively, an \emph{$F_\sigma$-subset}) of a topological space is a subset that can be represented as a countable intersection of open sets (respectively, a countable union of closed sets).  A subset of a topological space is \emph{meager} if it is a union of countably many nowhere dense sets. A \emph{comeager} set is a set whose complement is meager. Equivalently, a subset of a topological space is comeager if it is a countable intersection of sets with dense interior.

The Baire category theorem (see~\cite[Theorem~8.4]{Kec}) states that the intersection of any countable collection of dense open subsets of a Polish space $P$ is dense in $P$. In particular, comeager subsets of a Polish space are non-empty and the class of comeager subsets is closed under countable intersections.

\begin{prop}[{\cite[Theorem 3.11]{Kec}}]\label{Prop:Gd}
Let $P$ be a Polish space. Every $G_\delta$ subset of $P$ is a Polish space with respect to the induced topology.
\end{prop}

We will also need a result of Mycielski stating that for every meager relation on a Polish space, there exist $2^{\aleph_0}$ unrelated elements. For our purpose, it is more convenient to formulate this result in terms of comeger sets.

\begin{thm}[Mycielski, { \cite[Theorem 19.1]{Kec}}]\label{Thm:Myc}
Let $P$ be a non-empty perfect Polish space. For any comeager $R\subseteq P\times P$, there exists a subset $C\subseteq P$ of cardinality $2^{\aleph_0}$ such that $\{(x,y)\in C\times C \mid x\ne y\}\subseteq R$.
\end{thm}

\paragraph{2.3. Length functions and weights on groups.}
In this paper, we often define length functions by using weights on groups.

\begin{defn}\label{WF}
A \emph{weight function} on a group $G$ is a map  $\omega \colon G\to \NN\cup\{ 0\}$  such that $\omega (g)=0$ if and only if $g=1$.
If, in addition, $ \omega(g)=\omega(g^{-1})$ for all $g\in G$, we say that $\omega $ is \emph{symmetric}.

To every weight function $\omega$ on $G$ we associate a length function $\ell_\omega \in  L(G) $ defined by the formula
\begin{equation}\label{ellw}
\ell_\omega(g)=\min\left\{ \left. \sum_{i=1}^n \omega(x_i) \; \right|\; g=x_1^{\e_1}\cdots x_n^{\e_n},\; n\in \NN,\; x_1, \ldots , x_n\in G,\; \e_1, \ldots , \e_n =\pm 1 \right\}
\end{equation}
for every $g\in G$. Here the minimum is taken over all possible decompositions $g=x_1^{\e_1}\cdots x_n^{\e_n}$ of the element $g$. If the minimum in (\ref{ellw}) is attained at a certain decomposition, we call it a \emph{geodesic decomposition} of $g$. Note that if $\omega$ is symmetric, every $g\in G$ has a geodesic decomposition with $\e_1=\ldots =\e_n=1$.
\end{defn}

Recall that $L(G)$ denotes the space of all length functions $G\to \NN\cup\{ 0\}$ with the topology of pointwise convergence. Thus, the base of neighborhoods of a length function $\ell\in  L(G)$ is formed by the sets
\begin{equation}\label{Eq:WlF}
W(\ell, F)= \{ k\in L (G) \mid \forall \, f\in F\; k(f)=\ell(f)\},
\end{equation}
where $F$ ranges in the set of all finite subsets of $G$.

\begin{lem}\label{Lem:wM}
Let $G$ be a countably infinite group, $\ell\in  L(G)$, $F$ a finite subset of $G$ such that $F^{-1}=F$ and $1\in F$. Let $G\setminus F=\{ g_1, g_2, \ldots\}$. Given a sequence of natural numbers $M=(M_1, M_2, \ldots)$, we define a weight function
$\omega_M$ by the rule
$$
\omega_M (g)=\left\{
\begin{array}{cl}
    \ell(g), & {\rm if}\; g\in F \\
  M_i,  & {\rm if}\; g=g_i,\; i\in\NN.
\end{array}
\right.
$$
If $M_i> \max\{ \ell(f)\mid f\in F\}$ for all $i\in \NN$, then $\ell_{\omega_M}\in W(\ell, F)$.
\end{lem}

\begin{proof}
For every $f\in F$, we obviously have $\ell_{\omega_M}(f) \le  \omega_M(f) =\ell (f)< M_i$ for all $i$. Hence, elements of $G\setminus F$ cannot occur in a geodesic decomposition of $f$. Let $f=x_1^{\e_1}\cdots x_n^{\e_n}$ be a geodesic decomposition, where $x_1, \ldots, x_n \in F$ and $\e_1, \ldots , \e_n=\pm 1$. By the triangle inequality, we have
$$
\ell_{\omega_M} (f)=\sum_{i=1}^n \omega_M(x_i) = \sum_{i=1}^n \ell(x_i) \ge \ell (f).
$$
Thus, $\ell _{\omega_M}(f)=\ell(f)$ for all $f\in F$.
\end{proof}

We say that a function $f\colon G\to \NN\cup\{ 0\}$ is \emph{bounded} if $\sup\{ f(g) \mid g\in G\} <\infty$. Further, we say that $f$ is \emph{proper} if $|\{g \in G \mid f(g)<n\}|<\infty $ for every $n\in \NN$.

\begin{lem}\label{Lem:BU}
Let $G$ be a group and let $\omega\colon G\to \NN\cup\{ 0\}$ be a weight function. If $\omega$ is bounded (respectively, proper) then the length function $\ell_\omega$ is bounded (respectively, proper). In particular, if $G$ is infinite and $\omega$ is proper, then $\ell_\omega$ is unbounded.
\end{lem}

\begin{proof}
If $\omega $ is bounded, then so is $\ell_\omega$ since $\ell_\omega(g)\le \omega(g)$ for all $g\in G$ by definition. Further, assume that $\omega$ is proper. Then for every $n\in \NN$, only finitely many elements can occur in a geodesic decompositions of elements of length at most $n$. This obviously implies that $\ell_\omega$ is proper (and unbounded whenever $|G|=\infty$).
\end{proof}

\begin{cor}\label{Cor:BUdense}
If $G$ is a countably infinite group, then the sets of bounded and proper length functions are both dense in $L(G)$; in particular, the set of all unbounded length functions is dense in $L(G)$.
\end{cor}

\begin{proof}
We fix a length function $\ell\in  L(G)$ and consider any open neighborhood $U$ of $\ell$. By definition, there exists a finite subset $F\subseteq G$ such that $W(\ell, F)\subseteq U$. Without loss of generality, we can assume that $F=F^{-1}$ and $1\in F$. Applying Lemmas \ref{Lem:wM} and \ref{Lem:BU} to the sequence $M_i= \max\{ \ell(f)\mid f\in F\}+1$ (respectively, $M_i= \max\{ \ell(f)\mid f\in F\}+i$), we obtain a bounded (respectively, proper) length function in $W(\ell, F)$.
\end{proof}

\begin{proof}[Proof of Proposition \ref{Prop:P}]
The space $P=(\NN\cup \{ 0\}) ^G$ is Polish being a countable product of (discrete) Polish spaces. For every fixed $g,h\in G$, the set of all $\ell\in P$ satisfying conditions (L$_1$)--(L$_3$) in Definition~\ref{Def:LF} is open. Therefore, $L(G)$ is a $G_\delta $ subset of $P$. In particular, it is a Polish space by Proposition \ref{Prop:Gd}.  If $G$ is infinite, $L(G)$ is perfect by Corollary \ref{Cor:BUdense}.
\end{proof}

\section{Generic length functions and Cayley graphs}\label{Sec:CG}

\paragraph{3.1. The Extension Property.} We begin by recalling an equivalent definitions of the universal graph $U$ obtained by Moss in \cite[Theorem 4.1]{M}. Let $(S,\d_S)$, $(T,\d_T)$ be any metric spaces. We write $\bar s\cong\bar t$ for two tuples $\bar s=(s_1, \ldots, s_n)\in S^n$ and $\bar t=(t_1, \ldots, t_n)\in T^n$ if $\d_S(s_i, s_j)=\d_T(t_i, t_j)$ for all $i,j \in \{ 1, \ldots, n\}$. Moss proved that $U$ is the unique countable connected graph satisfying the following \emph{Extension Property}.

\begin{enumerate}
\item[(EP)] Let $n\in \NN$ and let $\bar a=(a_1, \ldots, a_n)$ be a tuple of vertices of $U$. Let $\bar b=(b_1, \ldots, b_n)$ be a tuple of vertices of another connected graph $\Gamma $ such that $\bar a\cong \bar b$. For any vertex $b$ of $\Gamma$, there exists a vertex $a$ of $U$ such that $(a_1, \ldots, a_n, a)\cong (b_1, \ldots, b_n, b)$.
\end{enumerate}

We will prove Theorem \ref{main1} by showing that (EP) holds for Cayley graphs associated to generic lengths functions on countable MIF groups. To this end, we need an auxiliary definition.

\begin{defn}
Let $G$ be a group, $\bar a=(a_1, \ldots, a_n)\in G^n$, $\ell\in  L(G)$.
We say that a tuple $\bar d=(d_1, \ldots, d_n)\in (\NN\cup\{ 0\})^n$ is \emph{$(\ell,\bar a)$-consistent}, if
\begin{equation}\label{Eq:cons}
d_i-d_j \le \ell(a_i^{-1} a_j)\le d_i+d_j
\end{equation}
for all $i,j\in \{1, \ldots, n\}$. For every tuples $\bar a\in G^n$ and $\bar d\in (\NN\cup \{ 0\})^n$, we denote by $D(\bar a, \bar d)$ the set of all $\ell\in  L(G)$ satisfying the following condition: if $\bar d$ is $(\ell, \bar a)$-consistent, then there is $g\in G$ such that
\begin{equation}\label{Eq:gai}
\ell(g^{-1} a_i)=d_i \;\;\; \forall \, i\in\{ 1, \ldots, n\}.
\end{equation}
\end{defn}

Note that $(\ell, \bar a)$-consistency of $\bar d$ is a necessary condition for the existence of an element $g$ satisfying (\ref{Eq:gai}).

\begin{lem}\label{Lem:D}
Let $G$ be a countable MIF group. For every $n\in \NN$ and every tuples $\bar a\in G^n$, $\bar d\in (\NN\cup\{ 0\})^n$, the set $D(\bar a, \bar d)$ is open and dense in $L(G)$.
\end{lem}

\begin{proof}
Throughout the proof, we fix some tuples $\bar a\in G^n$, $\bar d\in (\NN\cup\{ 0\})^n$, and let
$$
A=\{ a_i^{-1}a_j\mid i,j=1, \ldots , n\}.
$$

We first show that $D(\bar a, \bar d)$ is open. Let $\ell\in D(\bar a, \bar d)$. If $\bar d$ is not $(\ell,\bar a)$-consistent, then there are $i,j\in \{ 1, \ldots , n\}$ such that (\ref{Eq:cons}) fails. Obviously, the open neighborhood of $\ell$  consisting of all functions $k\in  L(G)$ such that $k (a_i^{-1}a_j)=\ell(a_i^{-1}a_j)$ is a subset of $D(\bar a, \bar d)$. Further, assume that $\bar d$ is $(\ell, \bar a)$-consistent. Since $\ell\in D(\bar a, \bar d)$, there exists $g\in G$ satisfying (\ref{Eq:gai}). Let $E=A\cup\{ g^{-1} a_i\mid i=1, \ldots, n\}$. Clearly, we have $W(\ell, E)\subseteq D(\bar a, \bar d)$, where $W(\ell, E)$ is defined by (\ref{Eq:WlF}).

Next, we show that $D(\bar a, \bar d)$ is dense in $L(G)$. It suffices to show that for every $\ell\in  L(G)$ and every finite subset $F\subseteq G$, the intersection $W(\ell, F)\cap  D(\bar a, \bar d)$ is non-empty. Without loss of generality, we can assume that $A\subseteq F$, $1\in F$, and $F=F^{-1}$.

If $\bar d$ is not $(\ell, \bar a)$-consistent, then $\ell\in W(\ell, F)\cap D(\bar a, \bar d)$ by definition. Henceforth, we assume that $\bar d$ is $(\ell, \bar a)$-consistent.
Let
$$
B=\bigcup_{i=1}^n a_iF
$$
and
$$
M=\max(\{ \ell(h)\mid h\in F\} \cup \{d_1, \ldots , d_n\}) +1.
$$
Consider the following subsets of $G\ast \langle x\rangle$:
$$
I_1=\{ xb^{-1} \mid b\in B\},
$$
$$
I_2= \{ f_1x^{\e_1}\ldots f_Mx^{\e_M} \mid \e_i\in \{ -1, 0, 1\}, \; f_i \in F\;\, \forall\, i=1, \ldots, M\}\setminus \{ 1\},
$$
and
$$I=I_1\cup I_2.$$
None of the elements of $I$ represents $1$ in $G\ast\langle x\rangle$. Therefore, $G\ast \langle x\rangle$ does not satisfy the universal sentence
$$
\forall \, g \; \left(\bigvee_{w(x)\in I} w(g)=1\right).
$$
By Proposition \ref{Prop:HO}, there exists $g\in G$ such that $w(g)\ne 1$ for all $w(x)\in I$. We fix any such an element $g$, let
$$
K=\{ (g^{-1}a_i)^{\pm 1}\mid i=1, \ldots , n\},
$$
and define the symmetric weight function
$$
\omega_g(h)=
\left\{
\begin{array}{cl}
  d_i, & {\rm if}\; h=(g^{-1}a_i)^{\pm 1}\; {\rm for\; some}\; i=1, \ldots, n, \\
 \ell(h), & {\rm if}\; h\in F,\\
 M, & {\rm if}\; h\notin K\cup F.
\end{array}
\right.
$$
Since $g$ does not satisfy any equation $w(x)=1$ for $w(x)\in I_1$, we have $g\notin B$. Together with the assumption $F=F^{-1}$, this ensures that $(g^{-1}a_i)^{\pm 1}\notin F$ for all $i$. Further, we have $g^{-1}a_i\ne (g^{-1}a_j)^{-1}$ for all $i$ and $j$; indeed, otherwise $g$ satisfies an equation of the form $a_j^{-1}xa_i^{-1}x=1$, whose left side belongs to $I_2$. Thus, $\omega_g$ is well-defined.

We will prove that
\begin{equation}\label{l_in_W}
\ell_{\omega_g}\in W(\ell, F)\cap D(\bar a, \bar d).
\end{equation}
Arguing by contradiction, assume that $\ell_{\omega_g}\notin W(\ell,F)$ or $\ell_{\omega_g}\notin D(\bar a, \bar d)$. We consider these possibilities separately.

\noindent{\it Case 1.} Suppose first that $\ell_{\omega_g}\notin W(\ell,F)$. Then $\ell_{\omega_g}(f)\ne \ell(f)$ for some $f\in F$. We will show that $w(g)=1$ for some $w(x)\in I_2$ in this case, thus reaching a contradiction.

Since $\ell_{\omega_g}(f)\le \omega_g(f)=\ell(f)$, we have $\ell_{\omega_g}(f)< \ell(f)$. This means that there exists a geodesic decomposition
\begin{equation}\label{Eq:fdec}
f=y_1\ldots y_m,
\end{equation}
where $y_1, \ldots, y_m\in G$, and
\begin{equation}\label{Eq:lgf}
\sum\limits_{i=1}^m \omega_g(y_i)< \ell(f).
\end{equation}
Since $\ell_{\omega_g}(f)\le \omega_g(f)< M$, we have $y_1, \ldots, y_m\in F\cup K$ and $m<M$. We additionally assume that (\ref{Eq:fdec}) has the least number of multiples $y_i\in K$ among all such (geodesic) decompositions. For each $i=1, \ldots, m$, we define $z_i\in G\ast \langle x\rangle$ by the rule
\begin{equation}\label{Eq:zdef}
z_i =
\left\{
\begin{array}{cl}
  y_i, &\; {\rm if}\; y_i\in F, \\
 x^{-1}a_{j}, &\; {\rm if}\; y_i=g^{-1}a_{j} \; {\rm for \; some} \; j=1, \ldots, n,\\
 a_{j}^{-1}x, &\; {\rm if}\; y_i=(g^{-1}a_{j})^{-1}=a_{j}^{-1}g \; {\rm for \; some} \; j=1, \ldots, n.
\end{array}
\right.
\end{equation}
Let $w(x)=f^{-1}z_1\ldots z_m$. Clearly, $g$ satisfies the equation $w(g)=1$. Since $m<M$, we obviously have $w(x)\in I_2\cup \{ 1\}$. It remains to show that $w(x)$ is a non-trivial  element of $G\ast \langle x\rangle$.

Assume first that $z_i\in G$ for all $i$. Then $y_i\in F$ for all $i$ and we obtain
$$
\sum\limits_{i=1}^m \omega_g(y_i) =\sum\limits_{i=1}^m \ell(y_i)\ge \ell (y_1\ldots y_m)=\ell(f),
$$
This contradicts (\ref{Eq:lgf}). Thus, $z_i\notin G$ for at least one $i$ and at least one of the letters $x$, $x^{-1}$ indeed occurs in the word $w(x)$.

The word $w(x)$ can still represent $1$ in $G\ast \langle x\rangle$ if all occurrences of $x$ and $x^{-1}$ cancel in the normal form of $w(x)$. To rule out this possibility, we first note that no subword $z_{i}z_{i+1}\ldots z_{i+k}$ ($k\ge 0$) of $w(x)$ can represent $1$ in $G\ast\langle x\rangle$; for otherwise, the decomposition $f=y_1\ldots y_m$ can be shortened by removing $y_iy_{i+1}\ldots y_{i+k}$, which contradicts the assumption that (\ref{Eq:fdec}) is geodesic. Therefore, the equality $w(x)=1$ in $G\ast \langle x\rangle$ implies the existence of  $i\in \{ 1, \ldots, m-1\}$ such that $z_i=a_{j_1}^{-1} x$ and $z_{i+1} =x^{-1}a_{j_2}$ for some $j_1, j_2\in \{ 1, \ldots, n\}$. In this case, we can replace the product $y_{i} y_{i+1}$ in (\ref{Eq:fdec}) with the single element $a_{j_1}^{-1} a_{j_2}\in A\subseteq F$. The obtained decomposition of $f$ contains fewer elements from $K$ and is also geodesic since
$$
\omega _g(a_{j_1}^{-1} a_{j_2})=\ell(a_{j_1}^{-1} a_{j_2})\le d_{j_1}+d_{j_2}= \omega_g(a_{j_1}^{-1}g)+ \omega_g(g^{-1}a_{j_2})= \omega_g(y_i) +\omega_g(y_{i+1}).
$$
by the right inequality in (\ref{Eq:cons}). This contradicts our choice of the decomposition (\ref{Eq:fdec}). Thus, $w(x)\ne 1$ in $G\ast \langle x\rangle$.

\noindent{\it Case 2.} Suppose now that $\ell_{\omega_g}\notin D(\bar a, \bar d)$ and $\ell_{\omega_g}(f)=\ell(f)$ for all $f\in F$. The latter assumption implies that $\bar d$ is $(\ell_{\omega_g}, \bar a)$-consistent. Therefore, we have $\ell_{\omega_g}(g^{-1}a_t) \ne d_t$ for some $t\in \{ 1, \ldots, n\}$. Since $\ell_{\omega_g}(g^{-1}a_t)\le \omega_g(g^{-1}a_t)= d_t$, we actually have $\ell_{\omega_g}(g^{-1}a_t) < d_t$ in this case. Therefore, there exists a geodesic decomposition
\begin{equation}\label{Eq:gadec}
g^{-1}a_t=y_0\ldots y_m,
\end{equation}
where $y_0, \ldots, y_m\in G$, and
\begin{equation}\label{Eq:lgga}
\sum\limits_{i=0}^m \omega_g(y_i)< d_t.
\end{equation}
As in Case 1, we show that $y_0, \ldots, y_m\in F\cup K$ and assume that (\ref{Eq:gadec}) has the least number of elements $y_i\in K$ among all such decompositions.

From this point, our argument is similar to that in Case 1. For every $i=0, \ldots, m$, we define $z_i$ by (\ref{Eq:zdef}) and consider the element $u(x)= a_t^{-1}x z_0\ldots z_m\in G\ast\langle x\rangle$. Clearly, $u(g)=1$ in $G$. By (\ref{Eq:lgga}), we have $m+1 <d_t<M$. Hence, $u(x)\in I_2\cup \{ 1\}$.  We will show that, in fact, $u(x)\ne 1$ in $G\ast \langle x\rangle$, i.e., $u(x)\in I_2$, thus reaching a contradiction.

Assume that all occurrences of the letter $x$ in the word $u(x)$ cancel after reducing it to the normal form. As above, we observe that for every $i\in \{ 0, \ldots, m\}$, we have $z_0\ldots z_i\ne 1$ in $G\ast\langle x\rangle $. Thus the first occurrence of $x$ in $u(x)$ cancels in the normal form of $u(x)$ only if $z_0= x^{-1}a_j$ for some $j$. Hence $y_0=g^{-1}a_j$. Plugging this in (\ref{Eq:gadec}), we obtain
\begin{equation}\label{Eq:ajat}
a_j^{-1}a_t= y_1\ldots y_m.
\end{equation}
Note that $\omega_g(y_0)=\omega_g(g^{-1}a_j)=d_j$. Combining the left inequality in (\ref{Eq:cons}) with (\ref{Eq:lgga}) and (\ref{Eq:ajat}), we obtain
$$\ell(a_t^{-1}a_j)  \ge  d_t - d_j > \sum\limits_{i=0}^m\omega _g(y_i) - d_j = \sum\limits_{i=1}^m\omega _g(y_i)\ge \ell(a_t^{-1}a_j).$$
This contradiction shows that $u(x)\ne 1$.

Summarizing the discussion above, we conclude that both Case 1 and Case 2 are impossible. Thus, we have (\ref{l_in_W}). Therefore, $D(\bar a, \bar d)$ is dense in $L(G)$.
\end{proof}

\paragraph{3.2. Genericity of word lengths on countable MIF groups.} Given a group $G$ and an element $g\in G\setminus\{ 1\}$, we define
$$
C(g)=\{ \ell\in  L(G) \mid \exists \, h\in G\;\, \ell(h)=\ell(g)-1  \wedge  \ell(h^{-1}g)=1\}.
$$
Let also
$$
C=\bigcap\limits_{g\in G\setminus\{ 1\}} C(g).
$$

\begin{lem}\label{Lem:Clf}
For any group $G$, every $\ell\in C$ is a word length.
\end{lem}
\begin{proof}
Let $\ell\in C$ and let $X_\ell=\{ x\in G \mid \ell(x)=1\}$. We denote by $|\cdot|_{X_\ell}$ the word length on the subgroup $\langle X_\ell\rangle$ with respect to $X_\ell$. Our goal is to show that $\langle X_\ell\rangle=G$ and $\ell(g)=|g|_{X_\ell}$ for all $g\in G$.

For every $g\in \langle X_\ell\rangle$, we have $\ell(g)\le |g|_{X_\ell}$ by the triangle inequality. Conversely, we will prove that $g\in \langle X_\ell\rangle$ and $|g|_{X_\ell}\le \ell(g)$ for every $g\in G$ by induction on $\ell(g)$. If $\ell(g)=0$, this is obvious. Suppose that $\ell(g)\ge 1$. Since $\ell\in C\subseteq C(g)$, there exists $h\in G$ such that $\ell(h)=\ell(g)-1$ and $h^{-1}g\in X_\ell$. By the inductive assumption, we have $h\in \langle X_\ell\rangle$ and $|h|_{X_\ell}=\ell(h)$. Hence, $g= h (h^{-1}g)\in \langle X_\ell\rangle$ and
$$
|g|_{X_\ell}= |h (h^{-1}g)|_{X_\ell} \le |h|_{X_\ell} + |h^{-1}g|_{X_\ell}= \ell(h) +1=\ell(g).
$$
\end{proof}

\begin{lem}\label{Lem:C}
For any countable MIF group $G$, the set $C$ is comeager in $L(G)$.
\end{lem}

\begin{proof}
It suffices to show that $C(g)$ is comeager in $L(G)$ for every $g\in G\setminus \{ 1\}$. Fix any $g\in G\setminus \{ 1\}$. Obviously, $C(g)$ is open. Further, let $\ell\in  L(G)$, $\bar a= ( 1, g)$, $\bar d=(\ell(g)-1, 1)$. For every $k\in W(\ell, \{ g\} )$, the tuple $\bar d$ is $(k, \bar a)$-consistent. Let $U$ be any open neighborhood of $\ell$ in $L(G)$. By Lemma \ref{Lem:D}, $U\cap W(\ell, g)$ contains a length function $k$ such that there exists $h\in G$ satisfying $k(h)=\ell(g)-1=k(g)-1$ and $k(h^{-1}g)=1$. In particular, $k\in C(g)\cap U$. Thus, $C(g)$ is dense in $L(G)$.
\end{proof}

\begin{proof}[Proof of Theorem \ref{main1}]
By Lemmas \ref{Lem:Clf} and \ref{Lem:C}, there is a comeager subset $C\subseteq L(G)$ such that every $\ell\in C$ is a word length. Let
$$
D=\bigcap_{n\in \NN}\bigcap_{\bar a\in G^n}\bigcap_{\bar d\in (\NN\cup\{ 0\})^n} D(\bar a, \bar d).
$$
By Lemma \ref{Lem:D}, $D$ is also comeager in $L(G)$. To prove the theorem, it suffices to show that the Cayey graph corresponding to every $\ell\in C\cap D$ satisfies the Extension Property (EP).

Fix any $\ell\in C\cap D$. Let $\Gamma $ be a connected graph, and let $\bar b=(b_1,\ldots, b_n)$ be a tuple of vertices of $\Gamma$ for some $n\in \NN$. For any vertex $b$ of $\Gamma$, we define
$$
d_i=\d_\Gamma(b, b_i),
$$
where $\d_\Gamma$ is the standard metric in $\Gamma$. Suppose that $\bar a=(a_1, \ldots, a_n)\in G^n$ and $\bar a \cong \bar b$ with respect to the distance induced by $\ell$ on $\bar a$; this means that $\ell (a_i^{-1}a_j)=\d_\Gamma (b_i, b_j)$ for every $i, j\in \{ 1, \ldots, n\}$. Note that $$d_i-d_j\le \d_\Gamma (b_i, b_j)\le d_i+d_j$$
by the triangle inequality. Therefore, the tuple $\bar d$ is $(\ell, a)$-consistent. Since $\ell \in D\subseteq D(\bar a, \bar d)$, there is $g\in G$ such that $\ell(g^{-1} a_i)=d_i$ for all $i$. Thus, the condition (EP) holds for the Caley graph corresponding to $\ell$ and the theorem follows.
\end{proof}

\begin{rem}\label{Rem:ab} A similar idea can be used to prove the analogue of Theorem \ref{main1} for the group $G=\ZZ$ or, more generally, for any non-trivial, torsion-free, abelian group. We leave this as an exercise for the interested reader.
\end{rem}

The conclusion of Theorem \ref{main1} fails for infinite groups in general. We discuss two examples.

\begin{ex}\label{Ex:ab}
Let $G=\ZZ/4\ZZ \times A$, where $A=(\ZZ/2\ZZ)^\infty$. Consider the length function $\ell \in L(G)$ defined by
\begin{equation}\label{Eq:defl}
\ell (g)=3\;\;\; \forall\, g\in G\setminus\{ 1\}.
\end{equation}
Let $g=(2,0)\in G$. We claim that the open neighborhood $W(\ell, \{ g\})$ contains no word lengths. Indeed, suppose that $k\in W(\ell, \{ g\})$ is a word length. Every generating set of $G$ must contain an element $x$ such that the natural projection $G\to \ZZ/4\ZZ$ sends $x$ to a generator of $\ZZ/4\ZZ$. For every such $x$, we have $k(x)= 1$ and $x^2=g$. Using the triangle inequality, we obtain $\ell (g)=k (g)=k(x^2) \le 2$, which contradicts (\ref{Eq:defl}).
\end{ex}

\begin{ex}\label{Ex:vc}
Let $G=\langle a,b \mid b^4=1,\; b^{-1}ab=a^{-1}\rangle $. It is easy to see that every element $g\in G$ can be written in the form $g=a^\alpha b^\beta$, where $\alpha\in \ZZ$ and $\beta\in \{ 0, 1,2,3\}$. Further, if $\beta$ is odd, we have $g^2=b^2$, which implies
$g^{-1}= b^{-2}g=(g^{-1}b^2)^{-1}$. If $\beta $ is even, we have $bgb^{-1}=g^{-1}$, which implies $g^{-1}b=bg=b^{-3}g=(g^{-1}b^3)^{-1}$. Thus, every $g\in G$ satisfies one of the equalities
\begin{equation}\label{Eq:g}
g^{-1}=(g^{-1}b^2)^{-1}\;\;\; {\rm or}\;\;\; g^{-1}b=(g^{-1}b^3)^{-1}.
\end{equation}
Let $\ell\in L(G)$ be a word length and let $\Delta $ be the corresponding Cayley graph. For every vertex $g$ of $\Delta$, we have
$$
\d_\Delta(g, 1)= \ell (g^{-1})=\ell ((g^{-1}b^2)^{-1})= \ell (g^{-1}b^2)=\d_\Delta (g, b^2)
$$
or
$$
\d_\Delta(g, b)= \ell (g^{-1}b)=\ell ((g^{-1}b^3)^{-1})= \ell (g^{-1}b^3)=\d_\Delta (g, b^3)
$$
by (\ref{Eq:g}). However, given $\Delta$, we can always construct a graph $\Gamma $ containing vertices $v_1, \ldots, v_4$ and $w_0$ such that $(v_1, v_2, v_3, v_4)\cong (1,b,b^2, b^3)$ (we think of $1$, $b$, $b^2$, and $b^3$ as vertices of $\Delta $ here), $\d_\Gamma(w_0, v_1)\ne \d_\Gamma(w_0, v_3)$, and $\d_\Gamma(w_0, v_2)\ne \d_\Gamma(w_0, v_4)$. Indeed, denote $m=\d_\Delta(1,b)$ and define $\Gamma$ by adding $m$ new vertices $w_0, w_1, \cdots, w_{m-1}$ to the vertex set $G$ of $\Delta$ and connecting $1$ and $w_0$, $w_i$ and $w_{i+1}$ for each $i\in\{0,\cdots,m-2\}$, and $w_{m-1}$ and $b$. In another words, we add to $\Delta$ a new path of length $m+1$ that connects $1$ and $b$. Note that for any $g,h\in G$, if a path $p$ from $g$ to $h$ in $\Gamma$ without backtracking contains some new vertex $w_i$, then $p$ contains all the new vertices $w_0, \cdots, w_{m-1}$. Hence, by $\d_\Delta(1,b)<m+1$, we have $\d_\Gamma(g,h)=\d_\Delta(g,h)$ for any $g,h\in G$. Similarly, any path $p$ from $w_0$ to $b^3$ in $\Gamma$ without backtracking either (i) passes $1$ or (ii) contains all of $w_1,\cdots,w_{m-1}$. Let $|p|$ be the length of $p$, then in case (i), we have
$$
|p|
\ge 1+\d_\Gamma(1,b^3)
= 1+\d_\Delta(1,b^3)
=1+\d_\Delta(1,b)
=1+m.
$$
In case (ii), we have $|p|\ge m+\d_\Gamma(b,b^3)\ge m+1$. Thus, we have $\d_\Gamma(w_0,b^3)>m\ge \d_\Gamma(w_0,b)$. Also, since $w_0$ is connected by an edge only to $1$ and $w_1$, we have $\d_\Gamma(w_0, 1)=1 <\d_\Gamma(w_0, b^2)$. Hence, vertices $v_1, v_2, v_3, v_4$ of $\Gamma$ defined by $v_i=b^{i-1}$ satisfy the above property for $\Gamma$ together with $w_0$. Therefore, $\Delta$ does not satisfy the Extension Property. Thus, no Cayley graph of the group $G$ is isomorphic to $U$.
\end{ex}

\paragraph{3.3. Bi-Lipschitz equivalence of length functions.}
For every group $G$, we define
$$R_1=\{ (\ell_1, \ell_2)\in L(G)\times L(G)\mid \ell_1\not \preccurlyeq \ell_2 \}$$
and
$$R_2=\{ (\ell_1, \ell_2)\in L(G)\times L(G)\mid \ell_2\not \preccurlyeq \ell_1 \}.$$
Clearly, $R_1\cap R_2$ is precisely the set of all pairs $(\ell_1, \ell_2)\in L(G)\times L(G)$ such that $\ell_1$ and $\ell_2$ are $\preccurlyeq$-incomparable.

\begin{lem}\label{Lem:R1R2}
For any countably infinite group $G$,  the sets $R_1$ and $R_2$ are dense $G_\delta$ (in particular, comeager) in $L(G)\times L(G)$.
\end{lem}

\begin{proof}
We will prove the claim for $R_1$, the proof for $R_2$ is symmetric. For any $C \in \mathbb{N}$, let
$$
R_1(C) = \{(\ell_1, \ell_2)\in L(G)\times L(G) \mid \exists g \in G,\, \ell_1(g)> C\ell_2(g).\}
$$
The set $R_1(C)$ is open by the definition of the topology on $L(G)$. Further, let $\ell_1, \ell_2\in L(G)$ and let $U$ and $V$ be open neighborhoods of $\ell_1$ and $\ell_2$, respectively. By Lemma \ref{Lem:BU}, there exists an unbounded $k_1\in U$ and bounded $k_2\in V$. Obviously, $(k_1, k_2)\in R_1(C)$. Thus, $R_1(C)$ is dense in $L(G)\times L(G)$. It remains to note that $R_1=\bigcap_{C\in \NN} R_1(C)$ and apply the Baire category theorem.
\end{proof}

\begin{proof}[Proof of Theorem \ref{main2}]
Let $D$ be a comeger subset of $L(G)$. By the Kuratowski-Ulam theorem,  $D\times D$ is a comeager subset of $L(G)\times L(G)$ (see Theorem 8.41 in \cite{Kec}). Combining this with Lemma \ref{Lem:R1R2}, we obtain that the set $R=(D\times D)\cap R_1\cap R_2$ is also comeager in $L(G)\times L(G)$. By Theorem \ref{Thm:Myc}, there exists a subset $C\subseteq L(G)$ of cardinality $2^{\aleph_0}$ such that $\{ (\ell_1, \ell_2)\in C\times C\mid \ell_1\ne \ell_2\}\subseteq R$. Clearly, $C\subseteq D$ and distinct elements of $C$ are $\preccurlyeq$-incomparable.
\end{proof}

The following lemma can easily be generalized to length functions obtained from cobounded actions on an arbitrary geodesic space. We restrict to the particular case of actions on Cayley graphs for simplicity. For a group $G$ generated by a set $X$, we denote by $Cay(G,X)$ the corresponding (unlabelled, unoriented) Cayley graph. Recall that the vertices of $G$ are elements of $G$ and two vertices $a, b\in G$ are connected by an edge if and only if $a^{-1}b\in X$. The action of $G$ on itself by left multiplication gives rise to a $G$-action on $Cay(G,X)$ by automorphisms. Let $\Gamma$ be a fixed graph. If $Cay(G,X)$ is isomorphic to $\Gamma$, every isomorphism $\iota\colon Cay(G,X)\to\Gamma$ uniquely defines a homomorphism $\alpha\colon G\to \Aut(\Gamma)$ such that the following diagram is commutative for every $g\in G$:
\begin{equation}\label{Eq:CD}
\begin{tikzcd}
Cay(G,X) \arrow{r}{\iota} \arrow{d}{g} & \Gamma \arrow{d}{\alpha(g)} \\%
Cay(G,X) \arrow{r}{\iota}& \Gamma
\end{tikzcd}
\end{equation}
Clearly, $\alpha $ is injective and the action of $\alpha(G)$ on the set of vertices of $\Gamma$ is regular. We call $\alpha$ a \emph{regular representation of $G$ by automorphisms of $\Gamma$}.

If $\iota_1, \iota_2\colon Cay(G,X)\to \Gamma$ are two isomorphisms, then there exists $t\in \Aut (\Gamma)$ such that $\iota _2=t\circ \iota_1$. Let $\alpha_1,\alpha_2\colon G\to Aut(\Gamma)$ be the regular representations corresponding to $\iota_1$ and $\iota_2$. By definition, we have
$$
\alpha_1(g)= \iota_1\circ g\circ \iota_1^{-1} = t^{-1}\circ \iota_2\circ g\circ \iota_2^{-1}\circ t = t^{-1} \alpha_2 (g) t
$$
for all $g\in G$. This prompts the following.

\begin{defn}
Two homomorphisms $\alpha_1, \alpha_2\colon G\to \Aut (\Gamma)$ are \emph{equivalent} if there exists $t\in \Aut(\Gamma)$ such that
\begin{equation}\label{Eq:conj}
\alpha_1(g)=t^{-1} \alpha_2(g) t\;\;\; \forall \, g\in G.
\end{equation}
\end{defn}

In particular, regular representations corresponding to distinct isomorphisms $Cay(G,X)\to \Gamma$ are equivalent.

\begin{lem}\label{Lem:conj}
Let $G$ be a group, $\ell_1$, $\ell_2$ word lengths on $G$. Suppose that the Caley graphs of $G$ associated to $\ell_1$ and $\ell_2$ are isomorphic to a certain graph $\Gamma$; that is, there exist graph isomorphisms $\iota_i\colon Cay(G, X_i)\to \Gamma$, where $X_i=\{ g\in G\mid \ell_i(g)=1\}$ for $i=1,2$. Let $\alpha_1, \alpha_2\colon G\to \Aut(\Gamma)$ be the corresponding regular representations. If $\alpha_1$ and $\alpha_2$ are equivalent, then $\ell_1$ and $\ell_2$ are bi-Lipschitz equivalent.
\end{lem}

\begin{proof}
For an automorphism $a\in \Aut(\Gamma)$ and a vertex $v$ of $\Gamma$, we denote by $av$ the image of $v$ under $a$. Let $v_i=\iota_i(1)$ for $i=1,2$. By definition, we have $\alpha_i(g) =\iota_i\circ g\circ \iota_i^{-1}$  (see (\ref{Eq:CD})). Therefore,
\begin{equation}\label{Eq:ai(g)}
\alpha_i(g)v_i = (\iota_i\circ g\circ \iota_i^{-1}) v_i = (\iota_i\circ g\circ \iota_i^{-1}) \iota_i(1)= \iota_i(g).
\end{equation}
We denote by $\d_{\ell_i}$ and $\d_\Gamma$ the left-invariant metric on $G$ corresponding to $\ell_i$ ($i=1,2$) and the metric on $\Gamma$, respectively. By (\ref{Eq:ai(g)}), we have
\begin{equation}\label{Eq:li(g)}
\ell_i(g) = \d_{\ell_i}(g, 1)= \d_\Gamma(\iota_i(g), \iota_i(1))= \d_\Gamma(\alpha_i(g)v_i, v_i)
\end{equation}

Let $t\in \Aut(\Gamma)$ be an element satisfying (\ref{Eq:conj}) and let $x\in X_2$. Combining (\ref{Eq:li(g)}) with (\ref{Eq:conj}) and the triangle inequality, we obtain
\[
\begin{split}
\ell_1(x) = & \d_\Gamma(\alpha_1(x)v_1, v_1)= \d_\Gamma(t^{-1}\alpha_2(x)tv_1, v_1) = \d_\Gamma(\alpha_2(x)tv_1, tv_1) \le \\
& \d_\Gamma(\alpha_2(x)tv_1, \alpha_2(x)v_2) + \d_\Gamma(\alpha_2(x)v_2, v_2) +\d_\Gamma(v_2, tv_1)=\\
&2\d_\Gamma (tv_1, v_2)+\ell_2(x) \le C \ell_2(x),
\end{split}
\]
where $C=2\d_\Gamma(tv_1, v_2)+ 1$. This obviously implies $\ell_1(g) \le C\ell_2(g)$ for all $g\in G$. Thus, $\ell_1\preccurlyeq \ell_2$. Similarly, we have $\ell_2\preccurlyeq \ell_1$.
\end{proof}

\begin{proof}[Proof of Corollary \ref{Cor: non-conjugate}]
Let $G$ be a countable MIF group. By Theorem \ref{main1}, there exists a comeager subset $C\subseteq L(G)$ such that every length function $\ell\in C$ is a word length and the corresponding Cayley graph is isomorphic to $U$. By Theorem \ref{main2}, $C$ contains $2^{\aleph_0}$ pairwise non-equivalent length functions. The regular representations $G\to Aut(U)$ corresponding to these functions are non-equivalent by  Lemma \ref{Lem:conj}.
\end{proof}

\section{Topological transitivity of the conjugacy action}\label{Sec:TT}

Our next goal is to address the question of whether the action of a group on the space of its length functions is topologically transitive. We begin by introducing auxiliary notation, which will be used throughout this section.

We write  $a^b=b^{-1}ab$ and $[a,b]=a^{-1}b^{-1}ab$ for elements $a$, $b$ of a group $G$. In this notation, we have the following obvious identities:
\begin{equation}\label{Eq:id}
(a^b)^c= a^{bc}, \;\;\; {\rm and}\;\;\; a^b=a[a,b]=[b,a^{-1}]a.
\end{equation}
For a subset $F\subseteq G$ and an element $g\in G$, we let
$$
F^g=\{ f^g\mid f\in F\}\;\;\; {\rm and}\;\;\; g^F=\{ g^f\mid f\in F\}.
$$

The \emph{conjugacy action} of a group $G$ on the set of all functions $\ell\colon G\to \NN\cup\{0\}$ is defined by the equation (\ref{Eq:CA}) for all $g\in G$. Clearly, $L(G)$ is a $G$-invariant subset of $(\NN\cup\{ 0\})^G$ and the action $G\curvearrowright L(G)$ is continuous. Given a subset $S\subseteq (\NN\cup\{ 0\})^G$, we denote by $g\circ S$ the set $\{ g\circ \ell \mid \ell \in S\}$.

\paragraph{4.1. The case of MIF groups.}
Recall that an action of a group $G$ on a topological space $X$ is \emph{topologically transitive} if for any non-empty open $U,V\subseteq X$, there exists $g\in G$ such that $gU\cap V \ne \emptyset$. We begin by proving the following.

\begin{thm}\label{Thm:tt}
For any countable MIF group $G$, the conjugacy action $G\curvearrowright  L(G)$ is topologically transitive.
\end{thm}

\begin{proof}
Let $U,V$ be any nonempty open subsets of $L(G)$. By the definition of the topology on $L(G)$, there exist $\ell_0\in U$, $\ell_1\in V$, and a finite subset $F$ of $G\setminus\{1\}$ such that $F=F^{-1}$ and $W(\ell_0,F)\subseteq U$ and $W(\ell_1,F)\subseteq V$ (see (\ref{Eq:WlF})). Consider the finite subset
\[I=I_1\cup I_2\cup I_3\subseteq G\ast\la x \ra\cong G\ast\ZZ,\]
where
\[I_1=\left\{f^{-1}xg{x^{-1}}\mid  f,g\in F\right\},\]
$$
I_2=\left\{f^{-1}a_1\cdots a_n\; \left|\;  f\in F,\; a_1,\cdots, a_n\in F\cup F^{x^{-1}}, \; \sum_{a_j\in F}\ell_0(a_j)+\sum_{a_j\in F^{x^{-1}}}\ell_1\left(a_j^x\right)<\ell_0(f)\right.\right\}
$$
$$
I_3=\left\{f^{-1}a_1\cdots a_n \;\left|\; f\in F,\; a_1,\cdots, a_n\in F^x\cup F,\; \sum_{a_j\in F^x}\ell_0\left(a_j^{x^{-1}}\right)+\sum_{a_j\in F}\ell_1(a_j)<\ell_1(f)\right.\right\}.
$$

We claim that the set $I$ does not contain the identity element of $G\ast\la x \ra $. Indeed, this is obvious for $I_1$ since $1\notin F$. Suppose for contradiction that some product $f^{-1}a_1\cdots a_n\in I_2$ represents $1$ in $G*\la x\ra$. Then $f=a_1\cdots a_n$ and all occurrences of $x$ and $x^{-1}$ must cancel in the normal form of $a_1 \ldots a_n$. It follows that there is a subsequence $a_{i_1},\cdots,a_{i_m}$ of $a_1,\cdots,a_n$ such that $\{a_{i_1},\cdots,a_{i_m}\}\subset F$ and $f=a_{i_1}\cdots a_{i_m}$. This implies
\[\ell_0(f)=\ell_0(a_{i_1}\cdots a_{i_m})\le \ell_0(a_{i_1})+\cdots+\ell_0(a_{i_m})
\le \sum_{a_j\in F}\ell_0(a_j),\]
which contradicts the condition $\sum\limits_{a_j\in F}\ell_0(a_j)+\sum\limits_{a_j\in F^{x^{-1}}}\ell_1\left(a_j^x\right)<\ell_0(f)$ in the definition of $I_2$. Thus, $1\notin I_2$. Similarly, we prove that $1\notin I_3$.

As in the previous sections, we adopt the notation $w(x)$ for elements of $G\ast \la x\ra$ and denote by $w(g)$ the image of $w(x)$ under the homomorphism $G\ast \la x \ra\to G$ that is identical on $G$ and sends $x$ to $g$. By Proposition \ref{Prop:HO}, there exists $g\in G$ such that $w(g)\ne 1$ for all $w(x)\in I$ . In particular, we have $F^{g^{-1}}\cap F=\emptyset$ by the definition of $I_1$. Hence, we can define a symmetric weight function $\omega \colon G\to \NN\cup\{0\}$ as follows:
\[\omega(a)=
\begin{cases}
0 &if\; a=1 \\
\ell_0(a) &if\;a\in F \\
\ell_1\left(a^g\right) &if\; a\in F^{g^{-1}} \\
N+1 & if \; a\notin F\cup F^{g^{-1}},
\end{cases}
\]
where $N=\max\limits_{g\in F}\{\ell_0(g),\ell_1(g)\}$.

We claim that $\ell_\omega (a)=\ell_0(a)$ for all $a\in F$. Indeed, for any $a\in F$, we have $\ell_\omega(a)\le \omega(a)=\ell_0(a)$. Arguing by contradiction, suppose that $\ell_\omega(a)<\ell_0(a)$. Then, there exist $a_1,\cdots,a_n\in G\setminus\{1\}$ such that $a=a_1\cdots a_n$ and $\omega(a_1)+\cdots+\omega(a_n)=\ell_\omega(a)<\ell_0(a)$. Since $\ell_0(a)\le N$, we have $a_1,\cdots,a_n\in F\cup F^{g^{-1}}$. Hence,
\[\ell_0(a)>\omega(a_1)+\cdots+\omega(a_n)=\sum_{a_j\in F}\ell_0(a_j)+\sum_{a_j\in F^{g^{-1}}}\ell_1\left(a_j^g\right).\]
It follows that $w(g)=1$ for some $w(x)\in I_2$. This contradicts the choice of $g$.

Similarly, we show that $g^{-1}\circ \ell_\omega (a)=\ell_1(a)$ for any $a\in F$. Indeed, we have $g^{-1}\circ\ell_\omega=\ell_{g^{-1}\circ \omega}$ and the claim follows from the assumption that $g$ does not satisfy any equation $w(x)=1$, where $w(x)\in I_3$. Thus, we have $\ell_\omega\in W(\ell_0,F)\cap g\circ W (\ell_1,F)$. In particular, $U\cap g\circ V\neq\emptyset$.
\end{proof}

\paragraph{4.2. A necessary condition for topological transitivity of $G\curvearrowright L(G)$.}
Our next goal is to provide some non-trivial examples of infinite groups for which the conjugacy action on the space of length functions is not topologically transitive. To this end, we first show that if $G\curvearrowright L(G)$ is topologically transitive, then $G$ satisfies a certain algebraic condition (see Proposition \ref{Prop:non tt}).

Recall that a group $G$ is said to satisfy the \emph{infinite conjugacy classes} condition (abbreviated \emph{ICC}) if for every $g\in G\setminus\{ 1\}$, the conjugacy class of $g$ in $G$ is infinite.

\begin{lem}\label{Lem:cofinite}
Let $G$ be an infinite group. If $A\subseteq G$ and $|G\setminus A|<\infty$, then $G=\la A\ra$.
\end{lem}

\begin{proof}
Suppose that there is $g\in G\setminus \la A\ra$. Then $g\la A\ra$ is an infinite set disjoint from $\la A\ra$, which contradicts the assumption that $A$ is cofinite in $G$.
\end{proof}

\begin{lem}\label{Lem:ICC}
Let $G$ be a countably infinite group. Suppose that $G$ contains an non-trivial element with finite conjugacy class. Then the conjugacy action $G\curvearrowright  L(G)$ is not topologically transitive.
\end{lem}

\begin{proof}
Let $x$ be an element of $G\setminus\{1\}$ such that the conjugacy class $x^G$ is finite, $X=G\setminus x^G$, $Y=G$. By Lemma \ref{Lem:cofinite}, $X$ is a generating set of $G$. Let $\ell_X$ and $\ell_Y$ be the word lengths with respect to $X$ and $Y$, respectively. The sets
\[U=\{\ell\in L(G)\mid \ell(g)=1 \; \forall g\in x^G \}\;\;\; {\rm and }\;\;\;  V=L(G)\setminus U\]
are $G$-invariant and clopen in $L(G)$ because $x^G$ is conjugacy-invariant and finite. Clearly, we have $\ell_X\in V$ and $\ell_Y\in U$ and, therefore, $U$ and $V$ are nonempty. However, we have $g\circ U\cap V=U\cap V=\emptyset$ for all $g\in G$.
\end{proof}

\begin{lem}\label{Lem:Neumann}
Let $G$ be an ICC group. For any finite subsets $A$, $B$ of $G\setminus \{1\}$, there exists $g \in G$ such that $A^g \cap B = \emptyset$.
\end{lem}

\begin{proof}
Given $a\in A$ and $b\in B$, we let $F_{a,b}=\{g\in G\mid  a^g=b\}$. Fix an arbitrary element of $f_{a,b}\in F_{a,b}$ if $F_{a,b}\ne \emptyset$. Obviously, we have $F_{a,b}=C_G(a)f_{a,b}$ where $C_G(a)$ is the centralizer of $a$ in $G$. Arguing by contradiction, suppose that $A^g \cap B \ne \emptyset$ for all $g\in G$; then $$G=\bigcup\limits_{a\in A, \; b\in B}F_{a,b} = \bigcup C_G(a)f_{a,b},$$ where the second union is taken over all $a\in A$, $b\in B$ such that $F_{a,b}\ne \emptyset$.

By Neumann's lemma (see \cite{Neu}), a group can be covered by a finite number of cosets of some subgroups only if one of the subgroups is of finite index. Therefore, there exists $a_0\in A$ such that $[G:C_G(a_0)]<\infty$. This implies $\left|a_0^G\right|=[G:C_G(a_0)]<\infty$. However, we have $a_0\neq 1$ since $1\notin A$, which contradicts the assumption that $G$ is ICC.
\end{proof}

\begin{prop}\label{Prop:non tt}
Let $G$ be a countably infinite group. Suppose that the conjugacy action $G\curvearrowright  L(G)$ is topologically transitive. Then for every finite index normal subgroup $H\le G$, the commutator subgroup $[H,H]$ is nontrivial ICC.
\end{prop}

\begin{proof}
Since the action $G \curvearrowright \mathcal{L}(G)$ is topologically transitive, $G$ is ICC by Lemma \ref{Lem:ICC}.
It is well-known and easy to prove that the ICC property passes to finite index subgroups. Thus, $H$ is  ICC  and, in particular, $[H,H]$ is nontrivial.

Suppose for contradiction that there exists $a \in [H,H]\setminus \{1\}$ such that the conjugacy class $a^{[H,H]}$ is finite. Let $E$ denote a transversal of $H$ in $G$; that is, $G=\bigsqcup_{e\in E} He$. The subgroup $[H,H]$ is normal in $G$ being a characteristic subgroup of the normal subgroup $H$. It follows that $(eae^{-1})^{[H,H]} = e(a^{[H,H]})e^{-1}$ for all $e\in E$. Thus, the union
\[ A= \bigcup_{e\in E} (eae^{-1})^{[H,H]} \]
is a finite subset of $H\setminus \{1\}$.

By Corollary \ref{Cor:BUdense}, there exists a proper length function $\ell_0\in L(G)$. In particular, the ball
$$
B_{\ell_0}(n)=\{ h\in H\mid \ell_0(h)\le n\}
$$
is finite for every $n\in \NN$. Let
\[ r= 2 \cdot max\{ \ell_0(e) \mid e\in E\} + \ell_0(a) + 2 \]
By Lemma \ref{Lem:Neumann}, there exists $h \in H$ such that
\begin{equation}\label{Eq:disj}
A^{h} \cap B_{\ell_0} (r) = \emptyset.
\end{equation}
We are going to use topological transitivity of the conjugacy action $G\curvearrowright L(G)$ to find an element $b\in A^{h} \cap B_{\ell_0} (r)$, thus reaching a contradiction.

Let
\[F = E \cup A \cup A^{h} \cup \{h\}\]
and let $\ell_1$ denote the length function on $G$ such that $\ell_1(g)=1$ for all $g\in G\setminus\{1\}$. By topological transitivity of the conjugacy action $G \curvearrowright L(G)$, there exist $g\in G$ and $\ell\in L(G)$ such that $\ell\in W(\ell_0, F)$ and $g\circ\ell \in W(\ell_1, F)$. Let $g=kf$, where $f\in E$ and $k\in H$, and let $$b=\left(a^{f^{-1}}\right)^{[k,h^{-1}]h}.$$ Using identities (\ref{Eq:id}), we obtain
$$
b^f=\left(a^{f^{-1}}\right)^{[k,h^{-1}]hf}=\left(a^{f^{-1}}\right)^{h^kf}=a^{f^{-1}h^kf}=a^{h^{kf}}=a^{h^g}.
$$
Since $\ell\in W(\ell_0, F)$, the restrictions of $\ell$ and $\ell_0$ to $F$ coincide. Further, the inclusions $g\circ\ell \in W(\ell_1, F)$ and $h\in F$ imply that $\ell(h^g)=(g\circ \ell) (h)= \ell_1(h)\le 1$. Note also that $b\in (a^{f^{-1}})^{[H,H]h}\subseteq A^h\subseteq F$ and $a,f\in F$. Therefore, we have
$$
\ell_0(b) = \ell(b) = \ell \left(\left(a^{h^g}\right)^{f^{-1}}\right)  \le 2\ell(f) + 2\ell(h^g) + \ell(a) \\
\le 2\ell_0(f) + 2 + \ell_0(a) \le r
$$
Therefore, $b\in A^h \cap B_{\ell_0} (r)$, which contradicts (\ref{Eq:disj}). Thus, $[H,H]$ is ICC.
\end{proof}

\begin{cor}\label{Cor:Sol}
For any countable group $G$ from the following classes, the conjugacy action $G\curvearrowright L(G)$ is not topologically transitive.
\begin{enumerate}
\item[(a)] Linear solvable groups over arbitrary fields.
\item[(b)] Nilpotent-by-abelian (in particular, metabelian) groups.
\end{enumerate}
\end{cor}

\begin{proof}
For any finite group, the action on the space of length functions is not topologically transitive, hence it suffices to prove the corollary for infinite groups. Suppose that $G$ is solvable and linear over a field. By the Lie-Kolchin theorem, $G$ contains a finite index normal subgroup $H$ such that $[H,H]$ is nilpotent. In particular, $[H,H]$ does not belong to the class of nontrivial ICC groups and we obtain (a) by Proposition \ref{Prop:non tt}. The proof of (b) is analogous.
\end{proof}

In light of Corollary \ref{Cor:Sol}, we would like to propose the following.

\begin{prob}
Does there exist a non-trivial, countable, solvable group $G$ such that the conjugacy action $G\curvearrowright L(G)$ is topologically transitive?
\end{prob}

\section{A topological zero-one law for length functions on MIF groups}\label{Sec:EE}

\paragraph{5.1. Infinitary logic and the Lopez-Escobar theorem for $L(G)$.}
Recall that $\L$ is the first order language in the signature (\ref{Eq:sig}). As usual, we denote by $\Lo$ the infinitary version of $\L$ that allows countable conjunctions and disjunctions (but only finite sequences of quantifiers). For details, we refer to \cite{Mar}.

The set of well-formed $\Lo$-formulas can be stratified as follows. Let $\mathcal{L}_0$ denote the set of all atomic formulas. If $\alpha>0$ is a limit ordinal, define $\mathcal{L}_\alpha=\bigcup_{\beta<\alpha}\mathcal{L}_\beta$. If $\alpha=\beta+1$ for some ordinal $\beta$, let $\mathcal{L}_\alpha$ be the set of all well-formed formulas of the following three types.
\begin{enumerate}
\item[(a)] $\forall\, x\; \phi$, $\exists\, x\;\phi$, $\neg\phi$, where $\phi \in \L_\beta$.
\item[(b)] $\bigwedge_{n\in \NN}\phi_n$, $\bigvee_{n\in \NN}\phi_n$, where $\phi_n\in\mathcal{L}_\beta$ for all $n\in \NN$.
\end{enumerate}
Thus, we have $\Lo=\bigcup_{\alpha<\omega_1}\mathcal{L}_\alpha$ where $\omega_1$ is the first uncountable ordinal. As usual, a well-formed formula is a called a \emph{sentence} if it does not contain free variables.

The well-known Lopez-Escobar theorem \cite{LE} states that every $\Lo$-sentence defines a Borel subset in the space of countable structures of any countable signature. We will need a version of this theorem for the space of length functions. More precisely, we fix a group $G$ and for every sentence $\sigma\in \Lo$ define its \emph{set of models}
$$
\Mod(\sigma)=\{\ell\in L(G)\mid G_\ell \models \sigma\},
$$
where $G_\ell$ is the $\L$-structure associated to $\ell$ as explained in the introduction.

\begin{lem}\label{Lem:LE}
For any countable ordinal $\alpha $, any formula $\sigma=\sigma(x_1,\cdots,x_n)\in \mathcal{L}_\alpha$  with finitely many free variables $x_1,\cdots,x_n$, and any $\bar{a}=(a_1,\cdots,a_n)\in G^n$, the set
\[\Mod(\sigma,\bar{a})= \{\ell\in L(G)\mid G_\ell\models \sigma(a_1,\cdots,a_n)\}\]
is Borel in $L(G)$.
\end{lem}

\begin{proof}
We proceed by transfinite induction on $\alpha$. If $\alpha=0$, $\sigma$ is an atomic formula. If it does not contain any predicate symbols, then $\sigma(a_1,\cdots,a_n)$ is a formula in the language of groups and $\Mod(\sigma,\bar{a})$ is either equal to $L(G)$ or empty depending on whether $G$ satisfies $\sigma(a_1, \ldots, a_n)$ or not. In particular, $\Mod(\sigma,\bar{a})$ is Borel. If $\sigma=L_i(t(x_1,\cdots,x_n))$ for some $i\in \NN\cup\{0\}$, where $t(x_1,\cdots,x_n)$ is a term, then for any $\bar{a}\in A^n$, $\Mod(\sigma,\bar{a})$ coincides with the set $\{\ell\in L(G) \mid\ell(t(a_1,\cdots,a_n))=i \}$, which is clopen in $L(G)$.

Further, assume that $\alpha>0$ and that the lemma holds for any $\beta<\alpha$. We fix some $\bar a\in G^n$. If $\alpha$ is a limit ordinal, the claim follows immediately from the inductive assumption. If $\alpha=\beta+1$, there are three cases to consider.

\noindent{\it Case 1.} First assume that $\sigma=\forall\, x\;\tau(x,x_1,\cdots,x_n)$ (or $\sigma= \exists\, x\; \tau(x,x_1,\cdots,x_n)$), where $\tau \in \mathcal{L}_\beta$. Then  $\Mod(\sigma,\bar{a})$ equals $\bigcap_{g\in G} \Mod(\tau,(g,\bar{a}))$ (respectively, $\bigcup_{g\in G} \Mod(\tau,(g,\bar{a}))$. Since $G$ is countable $\Mod(\tau,(b,\bar{a}))$ is Borel for any $g\in G$ by the inductive assumption, $\Mod(\sigma,\bar{a})$ is also Borel.

\noindent{\it Case 2.} Next, suppose that $\sigma=\neg \tau(x_1,\cdots,x_n)$, where $\tau \in \mathcal{L}_\beta$.  Then $\Mod(\sigma,\bar{a})$ coincides with $L(G)\setminus \Mod(\tau,\bar{a})$ and the latter set is Borel by the inductive assumption.

\noindent{\it Case 3.} Finally, let $\sigma=\bigwedge_{i\in \NN} \tau_i \left(x_{j_{(i,1)}},\cdots,x_{j_{(i,n_i)}}\right)$ or
$\sigma=\bigvee_{i\in \NN} \tau_i \left(x_{j_{(i,1)}},\cdots,x_{j_{(i,n_i)}}\right)$, where $\tau_i \in \mathcal{L}_\beta$ and $\left\{x_{j_{(i,1)}},\cdots,x_{j_{(i,n_i)}}\right\} \subset \{x_1,\cdots x_n\}$ for all $i$. In these cases, we have 
$$
\Mod(\sigma,\bar{a})=\bigcap_{i\in \NN}  \Mod\left(\tau_i,\left(a_{j_{(i,1)}},\cdots,a_{j_{(i,n_i)}}\right)\right)
$$ 
or 
$$
\Mod(\sigma,\bar{a})=\bigcup_{i\in \NN} \Mod\left(\tau_i,\left(a_{j_{(i,1)}},\cdots,a_{j_{(i,n_i)}}\right)\right),
$$
and the claim follows as in Case 1.
\end{proof}

\begin{prop}\label{Prop:LE}
Let $G$ be a countable group. For any sentence $\sigma \in \Lo$, $\Mod(\sigma)$ is a $G$-invariant Borel subset of $L(G)$.
\end{prop}

\begin{proof}
Fix any $\sigma \in \Lo$. It is easy to see that for every $\ell\in L(G)$ and $g\in G$, the $\L$-structures $G_\ell$ and $G_{g\circ \ell}$ are isomorphic. This implies that $\Mod(\sigma)$ is $G$-invariant for any sentence $\sigma \in \Lo$. That $\Mod(\sigma)$ is Borel follows immediately from Lemma \ref{Lem:LE}.
\end{proof}

\paragraph{5.2. A zero-one law and elementary equivalence of length functions.}
We will need the following standard fact.

\begin{thm}[ {\cite[Theorem 8.46]{Kec}}]
Suppose that a group $G$ acts topologically transitively by homeomorphisms on a Polish space $P$. Then every $G$-invariant Borel subset of $P$ is either meager or comeager.
\end{thm}

Combining this result with Theorem \ref{Thm:tt} and Proposition \ref{Prop:LE}, we immediately obtain the following.

\begin{thm}\label{Thm:0-1}
Let $G$ be a countable MIF group. For every $\Lo$-sentence $\sigma$, the set $\Mod(\sigma)$ is either meager or comeager.
\end{thm}

Theorem \ref{main3} easily follows from Theorem \ref{Thm:0-1}.

\begin{proof}[Proof of Theorem \ref{main3}]
Let $S$ be the set of all first order sentences in the signature $\{\cdot, L_0, L_1,\ldots\}$.
Define
$$
T=\{\sigma\in S\mid \Mod(\sigma)\textrm{ is comeager}\}
$$
Since $T$ is countable, its set of models  $C=\bigcap_{\sigma\in T} \Mod(\sigma)$ is comeager. Note that, for every sentence $\sigma\in S$, we have $\Mod(\sigma)\cup \Mod(\lnot \sigma)=L(G)$. By the Baire category theorem, $L(G)$ cannot be covered by two meager sets. Combining this with Theorem \ref{Thm:0-1}, we obtain that either $\sigma\in T$ or $\lnot \sigma\in T$, i.e., $T$ is complete. Since all models of a complete theory are elementarily equivalent, $G_{\ell_1}\equiv G_{\ell_2}$ for any $\ell_1, \ell_2\in C$.
\end{proof}

The set $T=T(G)$ defined in the proof of Theorem \ref{main3} is an invariant of $G$ that can be thought of as the elementary theory of a generic length function. Note that the equality $T(G)=T(H)$ for two groups $G$ and $H$ implies that $G$ and $H$ are elementarily equivalent (in the language of groups) since the group multiplication is a part of the signature. We conclude with the following.

\begin{prob}
Let $G$ and $H$ be countable MIF groups. Does $G\equiv H$ imply $T(G)=T(H)$? What about the case when $G$ and $H$ are free groups?
\end{prob}

\vspace{5mm}
\noindent \textbf{Andrew Jarnevic: } SC 1326, Department of Mathematics, Vanderbilt University, Nashville 37240, U.S.A.

\noindent E-mail: \emph{andrew.n.jarnevic@vanderbilt.edu}

\smallskip

\noindent \textbf{Denis Osin: } SC 1326, Department of Mathematics, Vanderbilt University, Nashville 37240, U.S.A.

\noindent E-mail: \emph{denis.v.osin@vanderbilt.edu}

\smallskip

\noindent \textbf{Koichi Oyakawa:} SC 1326, Department of Mathematics, Vanderbilt University, Nashville 37240, U.S.A.

\noindent E-mail: \emph{koichi.oyakawa@vanderbilt.edu}

\end{document}